\documentclass{amsart}
\usepackage{amssymb,latexsym, amsmath, amscd, array, graphicx}

\swapnumbers
\numberwithin{equation}{section}

\theoremstyle{plain}

\newtheorem{thm}{Theorem}[section]

\newtheorem{theorem}[thm]{Theorem}

\theoremstyle{definition}

\newtheorem{rem}[thm]{Remark}
\newtheorem{remark[thm]}{Remark}

\def\Z{{\mathbb Z}}

\def\R{{\mathbb R}}

\def\1{\hbox{\rm\rlap {1}\hskip.03in{\rom I}}}
\def\Bbbone{{\rm1\mathchoice{\kern-0.25em}{\kern-0.25em}
{\kern-0.2em}{\kern-0.2em}I}}

\long\def\forget#1\forgotten{} %

\begin{document}
\title[On Levin's generalization of the plus construction]
{On Levin's generalization of the plus construction}

\author[A.~Dranishnikov]
{Alexander Dranishnikov}
\address{A. Dranishnikov, Department of Mathematics, University
of Florida, 358 Little Hall, Gainesville, FL 32611-8105, USA}
\email{dranish@math.ufl.edu}

\maketitle

The plus construction was introduced by Kervaire~\cite{Ke} and Quillen~\cite{Q}. In a simple version it states the following:

{\bf Plus Construction Theorem.}~\cite{Ha} {\em Every connected CW complex $K$ with the perfect fundamental group $\pi_1(K)$ is contained in a simply connected
CW complex $K^+$ such that inclusion homomorphisms $H_i(K)\to H_i(K^+)$ are isomorphisms for all $i$. Moreover, $K^+$ can be obtained from $K$ by attaching cells of dimensions 2 and 3.}

\

Recently Michael Levin found a nice elementary generalization of the Plus Construction Theorem (Proposition 4.4~\cite{Le}). 
As far as I know his generalization was not known before. Since many algebraic topologists refused to believe in this remarkable fact, I decided to write a promotion note.

\

\begin{theorem} [M. Levin] For every connected CW complex $K$ there is a simply connected CW complex $K^+$ obtained from $K$
by attaching cells of dimension 2 and 3 such that the inclusion $K\to K^+$ induces isomorphisms of homology groups in dimension $>1$.
\end{theorem}
\begin{proof}
We attached 2-cells to $K$ to obtain a simply connected complex $K'$. Let $p:K'\to K'/K$ be the collapsing map.
Since $K'/K=\vee S^2$ is the wedge of spheres, the homoptopy group $\pi_2(K'/K)$ is free as well as its subgroup $p_{\#}(\pi_2(K'))=im(p_{\#})$.
We take a section $s:im(p_{\#})\to \pi_2(K')$ and fix a basis $\{[\phi_j]\}$ for $s(im(p_{\#}))\cong\oplus_J\Z$
where $\phi_j:S^2\to K'$. Then we attach
3-cells to $K'$ along the maps $\phi_j$ to obtain $K^+$. Note that $H_3(K^+,K')=\oplus_J\Z$ and in view of the Hurewicz isomorphism the through homomorphism
$$
\begin{CD}
H_3(K^+,K') @>\partial>> H_2(K') @>p_*>> H_2(K'/K)
\end{CD}
$$
is injective. It
takes $H_3(K^+,K')=\oplus_J\Z$ isomorphically onto the image $p_*(H_2(K'))=im(p_*)=\oplus_J\Z$. 

Note that the above composition $p_*\partial$ is the connecting homomorphism $d$ in the homology exact sequence
of the triple  $(K^+,K',K)$,
$$
\dots\to H_3(K',K) \to H_3(K^+,K) \to H_3(K^+,K') \stackrel{d}{\to} H_2(K',K) \stackrel{i}\to H_2(K^+,K)\to\dots
$$
Then the exact sequence of the triple and the obvious equality $H_3(K',K)=0$ imply that
$H_3(K^+,K)=0$. Since $H_i(K^+,K)=0$ for $i>3$ by dimensional reasons, the homology exact sequence of the pair 
$(K^+,K)$ 
$$
\dots \to H_{i+1}(K^+,K)\to H_i(K)\stackrel{i'}\to H_i(K^+)\stackrel{j'}\to H_i(K^+,K)\to\dots
$$
implies that the inclusion homomorphism $H_i(K)\to H_i(K^+)$ is an isomorphism for $i\ge 3$
and is a monomorphism for $i=2$.
Also the exact sequence of the triple $(K^+,K',K)$ implies that the inclusion homomorphism $i:H_2(K',K)\to H_2(K^+,K)$
takes $im(p_*)=im(d)$ to 0.
Then the commutative diagram
$$
\begin{CD}
H_2(K^+) @>j'>> H_2(K^+,K)\\
@Ai'AA @AiAA\\
H_2(K') @>p_*>> H_2(K',K)\\
\end{CD}
$$
and the fact that $i'$ is surjective for 2-dimensional homology imply that $j'=0$. Therefore, $H_2(K)\to H_2(K^+)$ is an epimorphism and, hence, an isomorphism. 
\end{proof}

\begin{rem}
If $\pi_1(K)$ is perfect, hence $H_1(K)=0$ and Levin's plus construction implies the classical one.
\end{rem}

\begin{rem}
The original Levin's proof was presented on the chain level. Here we gave the proof in the diagram chasing language which is perhaps more appealing to algebraic topologists.
\end{rem}

\begin{rem}
The original plus construction in the full generality kills a given perfect normal subgroup of $\pi_1(K)$ preserving homology groups. We note that a subgroup version of the plus construction does not admit Levin's type generalization to an arbitrary normal subgroup. As it was noticed in~\cite{Ha} the commutator subgroup of the fundamental group of
$K=S^1\vee S^1$ cannot be killed without creating a new 2-dimensional homology.
\end{rem}

\end{document}